\documentclass[11pt,reqno,tbtags]{amsart}

\usepackage{amsthm,amssymb,amsfonts,amsmath,mathrsfs}
\usepackage[numbers, sort&compress]{natbib}
\usepackage{caption,graphicx}
\usepackage[labelformat=simple]{subcaption}
\usepackage[rmargin=40mm,lmargin=40mm,bmargin=30mm,tmargin=30mm]{geometry}
\usepackage{setspace}
\usepackage{paralist}
\usepackage[usenames,dvipsnames,svgnames,table]{xcolor}
\usepackage{hyperref}
\hypersetup{
	unicode=true,
	colorlinks=true,
	filecolor={black},
	linkcolor={black},
	citecolor={black},
	urlcolor={blue!60!black},
	pdftitle={A logarithmic bound for the chromatic number of the associahedron},
	pdfauthor={Addario-Berry, Reed, Scott, Wood},	
	bookmarks=false,        
	unicode=false,          
	pdftoolbar=true,        
	pdfmenubar=true,        
	pdffitwindow=true,      
	pdfnewwindow=true,      
}
\usepackage{cleveref}

\newtheorem{thm}{Theorem}

\newtheorem{prop}[thm]{Proposition}

\providecommand{\N}{}
\renewcommand{\N}{{\mathbb N}}

\newcommand\cA{\mathcal A}

\newcommand{\defn}[1]{\textcolor{Maroon}{\emph{#1}}}

\renewcommand{\geq}{\geqslant}
\renewcommand{\ge}{\geqslant}
\renewcommand{\leq}{\leqslant}
\renewcommand{\le}{\leqslant}

\newcommand{\doi}[1]{\url{https://dx.doi.org/#1}}
\begin{document}

\title[A logarithmic bound for the chromatic number of the associahedron]{A logarithmic bound for the chromatic number\\ of the associahedron} 
%%%%%
\author{Louigi Addario-Berry}
\address[L. Addario-Berry]{Department of Mathematics and Statistics, McGill University, Montr\'eal, Qu\'ebec, Canada}
\email{louigi@gmail.com}
\thanks{Louigi Addario-Berry is supported by an NSERC Discovery Grant and by the Canada Research Chairs program}
%%%%%
\author{Bruce Reed}
\address[B. Reed]{Mathematical Institute, Academia Sinica, Taiwan}
\email{bruce.al.reed@gmail.com}
%%%%%
\author{Alex Scott}
\address[A. Scott]{Mathematical Institute, University of Oxford, Oxford OX2 6GG, United Kingdom}
\email{scott@maths.ox.ac.uk}
\thanks{Alex Scott is supported by EPSRC grant EP/X013642/1.}
%%%%%
\author{David R. Wood}
\address[D. R. Wood]{School of Mathematics, Monash University, Melbourne, Australia}
\email{david.wood@monash.edu}
\thanks{David Wood is supported by the Australian Research Council}
%%%%%

\date{November 21, 2018; revised \today}
%\urladdrx{http://www.problab.ca/louigi/}

%\keywords{<keywords>}

%{60C05 (68P10,68W40)} %%{Primary: <subject>; Secondary: <subject>}
\begin{abstract} 
We show that the chromatic number of the $n$-dimensional associahedron grows at most logarithmically with $n$, improving a bound from and proving a conjecture of Fabila-Monroy~et~al.~(2009).
%[\emph{Discrete Math.\ Theor.\ Comput.\ Sci.}, 2009].
\end{abstract}

\maketitle

%%%%%%%%%%%%%%%%%%%%%%%%%%%%
% INTRODUCTION
%%%%%%%%%%%%%%%%%%%%%%%%%%%%
\section{Introduction} 
\label{sec:intro} 

The associahedron $\cA_n$ is an $(n-3)$-dimensional convex polytope that arises in numerous branches of mathematics, including algebraic combinatorics~\cite{Stasheff63,LR98,CFZ02,HLT11} and discrete geometry~\cite{BFS90,PS12, PS15}. Associahedra are also called Stasheff polytopes after the work of Stasheff~\cite{Stasheff63}, following earlier work by Tamari~\citep{Tamari51}. We are only interested in the 1-skeleton of the associahedron, so we consider it as a graph, defined as follows. 

The elements of the associahedron $\cA_n$ are triangulations $T$ of the convex $n$-gon with vertices labeled by $\{0,\ldots,n-1\}$ in clockwise order. For any such triangulation $T$, we always denote triangles of $T$ by the sequence $ABC$ of their vertices, ordered so that $A < B < C$.  We write 
$E(T)$ for the  set of edges contained in $T$.
Every triangulation $T$ of $\cA_n$ contains the edges $01,12,\ldots, (n-1)0$; we refer to these as \defn{boundary edges}. For $T$ in $\cA_n$, each non-boundary edge $e \in E(T)$ is contained in a unique quadrilateral $Q=Q_T(e)=ABCD$ with $A < B < C < D$; we always list the vertices of quadrilaterals in increasing order. \defn{Flipping} the edge $e$ means replacing $e$ by the other diagonal of $Q$; see Figure~\ref{fig:basic_quads}. Two triangulations $T,T'$ in $\cA_n$ are adjacent in $\cA_n$ if they may be obtained from one another by a single flip. 

%Figure~\ref{fig:basic_quads}(\subref{fig:tri_quad}) test \ref{fig:tri_quad} test 

\begin{figure}[!ht]
\begin{subfigure}[b]{0.45\textwidth}
\begin{centering}
\includegraphics[width=0.65\textwidth,page=1]{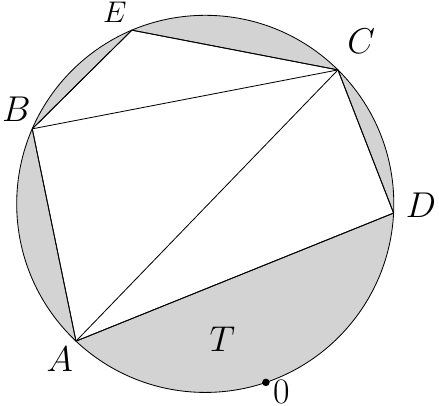}
\caption{A portion of a triangulation $T$. }
\label{fig:tri_quad}
\end{centering}
\end{subfigure}%
\quad
\begin{subfigure}[b]{0.45\textwidth}
\begin{centering}
\includegraphics[width=0.65\textwidth,page=2]{edge_flip.pdf}
\caption{A portion of the triangulation $T'$ formed from $T$ by flipping $AC$. }
\label{fig:tri_quad_flip}
\end{centering}
\end{subfigure}%
\caption{Portions of two adjacent triangulations of an $n$-gon.} 
\label{fig:basic_quads}
\end{figure}

Graph-theoretic properties of associahedra have been well-studied. For example, it is easily seen that $\cA_n$ is $(n-3)$-regular. Lucas~\cite{Lucas87} and  Hurtado and Noy~\citep{HN99} both proved that $\cA_n$ is Hamiltonian. Hurtado and Noy~\citep{HN99} also showed that $\cA_n$ has connectivity $n-3$, as well as determining its automorphism group. Parlier and Petri proved bounds on the genus of $\cA_n$. The diameter of $\cA_n$ and several related questions have been studied extensively~\cite{STT88,CP16,Pournin14,AMP15,CM18,Pournin19,Pournin17,CLP18,BLPV18,MP14,Dehornoy10}. Sleator~et~al.~\cite{STT88} proved that the diameter equals $2n-10$ for sufficiently large $n$, and recently Pournin~\cite{Pournin14} showed that $2n-10$ is the answer for $n>12$. Several authors \cite{MRS99,EF23,MT97} studied random walks in $\cA_n$.  

This paper studies the chromatic number of $\cA_n$, a quantity which was first considered by Fabila-Monroy~et~al.~\cite{FFHHUW09}. That work gave an explicit $\lceil{\frac{n}{2}\rceil}$-colouring of $\cA_n$, and observed that $\chi(\cA_n) \in O(n/\log n)$, based on the result of Johansson~\cite{johansson94} which says that every triangle-free graph with maximum degree $\Delta$ is $O(\Delta/\log\Delta)$-colourable. No non-constant lower bound on $\chi(\cA_n)$ is known. Indeed, the best known lower bound is $\chi(\cA_{10}) \geq 4$ [private communication, Ruy Fabila-Monroy]. Fabila-Monroy~et~al.~\cite{FFHHUW09} conjectured a $O(\log n)$ upper bound. We prove this conjecture.

\begin{thm}\label{thm:main}
$\chi(\cA_n) \in O(\log n)$. 
\end{thm}

\section{The Proof}

We prove Theorem~\ref{thm:main} by tracking how several carefully chosen properties of triangulations change when an edge is flipped. To see how this yields a route to bounding the chromatic number of $\cA_n$, first recall that if $f$ is a graph homomorphism from $\cA_n$ to some graph $G$, which is to say that $f:V(\cA_n) \to V(G)$ is adjacency-preserving, then $\chi(\cA_n) \le \chi(G)$. This fact may be generalized as follows. Suppose that $(G_i)_{i \in I}$ is a finite set of graphs and $(f_i: V(\cA_n) \to V(G_i))_{i \in I}$ are functions such that for all adjacent triangulations $T,T'$ in $\cA_n$, there exists $i \in I$ for which $f_i(T)$ and $f_i(T')$ are adjacent in $G_i$. For each $i \in I$, let $\kappa_i$ be a proper colouring of $G_i$ with $\chi(G_i)$ colours, and colour each $T$ in $\cA_n$ with the vector $(\kappa_i(f_i(T)))_{i \in I}$. If $T$ and $T'$ are adjacent in $\cA_n$, then $(\kappa_i(f_i(T)))_{i \in I}$ and $(\kappa_i(f_i(T')))_{i \in I}$ differ in at least one coordinate. Thus this is a proper colouring of $\cA_n$, and $\chi(\cA_n) \le \prod_{i \in I} \chi(G_i)$. The remainder of the paper is devoted to defining the five functions that we use (see \eqref{TheColouring}), and showing they have the requisite properties. 

Two fundamental notions that we use are the \defn{type} of a quadrilateral and the \defn{scale} of an edge.
For a quadrilateral $Q=Q_T(e)=ABCD$ contained within triangulation $T$, 
we say $Q$ is \defn{type-1} if $e=AC$, and otherwise say $Q$ is \defn{type-2}; 
we say that an edge $e$ is \defn{type-1} or \defn{type-2} according to the type of the quadrilateral  $Q_T(e)$.  
For example,  in Figure~\ref{fig:tri_quad},
$Q_T=ABCD$ is typ]e-1 and $Q_T(BC)=ABEC$ is type-2,
and in Figure~\ref{fig:tri_quad_flip}, 
$Q_{T'}(BC)=BECD$ is type-1 and $Q_{T'}(BD)=ABCD$ is type-2.

Fix an integer $\alpha \geq 3$ to be chosen later (in fact we end up taking $\alpha=3$). For an edge $e=UV$, define the \defn{scale} of $e$ to be 
\[
\sigma_e := \lceil \log_\alpha |U-V|\rceil \in \{0,1,\ldots,\lceil \log_\alpha(n-1)\rceil\}.\] 
Note that $\sigma_e=0$ if and only if $e$ is a boundary edge. 
The scales of the edges incident to triangles within a fixed quadrilateral $Q$ are a key input to the functions we define. 

%\sss{p. 2, l. 42: I am slightly puzzled that in a triangle ABC, the edges AB, BC and AC
%are called left $\ell$, middle $m$ and right $r$. I would have called AB $\leftrightarrow \ell$, AC $\leftrightarrow m$ and BC $\leftrightarrow r$. Everything would have then be more symmetric.}

We first consider the effect of edge flips on triangles $ABC$, where two of the three incident edges have the same scale. 
If $AB$ (resp.~$BC$, $AC$) is the unique edge whose scale is different from the others, then we say $ABC$ is a \defn{type-$\ell$} (resp.\ \defn{type-$m$}, \defn{type-$r$}) triangle. If all three edges have the same scale, then we say $ABC$ is a \defn{type-$z$} triangle. Let $(\ell_T,m_T,r_T,z_T)$ be the vector counting the number of type-$\ell$, type-$m$, type-$r$ and type-$z$ triangles in $T$. 

%\comment{DW: Do we ignore triangles where the three edges have three different scales?} \comment{DW: yes}

For the remainder of the paper, we fix a triangulation $T$, and consider the effect of flipping an edge $e=AC$ within a type-1 quadrilateral $Q_T(e)=ABCD$, to form another triangulation $T'$. 

\begin{prop}\label{prop:2_same}
If  $\sigma_{AC}$,  $\sigma_{BD}$ and $\sigma_{BC}$ are all different, then $(\ell_{T'},m_{T'},r_{T'},z_{T'})\ne (\ell_T,m_T,r_T,z_T)$. 
\end{prop}
 
\begin{proof} 
By assumption,  $\sigma_{BD}\ne \sigma_{AC}$ and $\sigma_{BC} < \min(\sigma_{AC},\sigma_{BD})$. 
We argue by contradiction. To this end, suppose that $(\ell_{T'},m_{T'},r_{T'},z_{T'})=(\ell_T,m_T,r_T,z_T)$. 
Since $\alpha\geq 3$, 
\[
\log_\alpha (D-A) \le \log_\alpha (3\max(B-A,C-B,D-C)) \le 1+ \log_\alpha \max(B-A,C-B,D-C). 
\]
Taking ceilings, it follows that 
\begin{equation}\label{eq:tight}
\sigma_{AD} \le 1+ \max(\sigma_{AB},\sigma_{BC},\sigma_{CD}).
\end{equation}
The preceding equation requires one of three inequalities to hold; the next three paragraphs treat the possibilities one at a time. 

Suppose that $\sigma_{AB} = \max(\sigma_{AB},\sigma_{BC},\sigma_{CD})$. Using \eqref{eq:tight} and the fact that $\sigma_{AB}\le \sigma_{AC}\le\sigma_{AD}$, we find that either $\sigma_{AC}=\sigma_{AB}$ or 
$\sigma_{AC}=\sigma_{AB}+1=\sigma_{AD}$. 

%\comment{DW: How do we write Figure~2(A) (as in the caption) instead of Figure~\ref{fig:tri_type_lz}? I actually prefer to write Figure~2(a), but what we write in the text should match the caption.} \comment{LAB: done.}

If  $\sigma_{AC}=\sigma_{AB}$, as in Figure~\ref{fig:tri_type_lz}, then $ABC$ is a type-$m$ triangle so, since we assume the triangle type vector is unchanged by flipping edge $AC$, either $ABD$ or $BCD$ is also type-$m$. 
$BCD$ is not type-$m$, as $\sigma_{BC}<\sigma_{BD}$ by assumption, so $ABD$ is type-$m$ and hence $\sigma_{AB}=\sigma_{AD}$.  But then $ABC$ and $ACD$ are both type-$m$, which gives a contradiction as $BCD$ is not.  

If $\sigma_{AC}=\sigma_{AB}+1$, as in Figure~\ref{fig:tri_type_lm}, then $ACD$ is type-$m$ or type-$z$, so either $ABD$ or $BCD$ is type-$m$ or type-$z$.  But $BCD$ is neither, as $\sigma_{BC} < \sigma_{BD}$, and $ABD$ is neither as $\sigma_{AB}\ne\sigma_{AD}$.

Next suppose that $\sigma_{BC}=\max(\sigma_{AB},\sigma_{BC},\sigma_{CD})$, as in Figure~\ref{fig:tri_type_mz}. 
Since $\sigma_{BC}\ne \sigma_{AC}$ and $\sigma_{BC}\ne \sigma_{BD}$ by assumption, 
and all scales are at most $\sigma_{AD}$, it must be that 
$\sigma_{AC}=\sigma_{BD}=\sigma_{AD}$; but this is ruled out by assumption. 

Finally, suppose that $\sigma_{CD} = \max(\sigma_{AB},\sigma_{BC},\sigma_{CD})$. 
This case is the same as the first case, as we can apply the argument to a reversed copy of the associahedron (which exchanges type-$\ell$ and type-$m$ triangles, while leaving the other two types invariant).
\end{proof}
\begin{figure}[ht]
\begin{subfigure}[b]{0.3\textwidth}
\begin{centering}
		\includegraphics[width=0.9\textwidth,page=3]{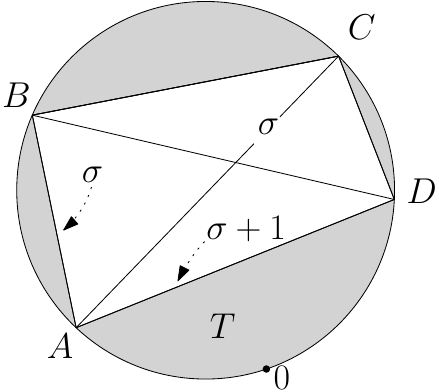}
                \caption{$\sigma_{AB}=\sigma=\sigma_{AC}$.}
                \label{fig:tri_type_lz}
\end{centering}
\end{subfigure}%
\quad
\begin{subfigure}[b]{0.3\textwidth}
\begin{centering}
		\includegraphics[width=0.9\textwidth,page=4]{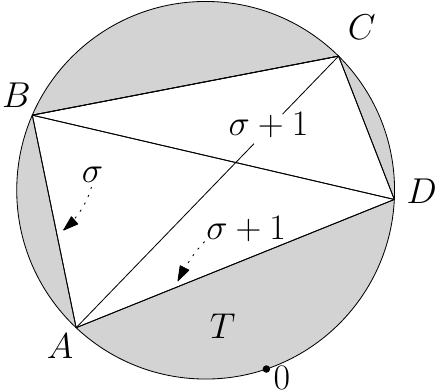}
                \caption{$\sigma_{AB}=\sigma,\sigma_{AC}=\sigma+1$.}
                \label{fig:tri_type_lm}
\end{centering}
\end{subfigure}%
\quad
\begin{subfigure}[b]{0.3\textwidth}
\begin{centering}
		\includegraphics[width=0.9\textwidth,page=5]{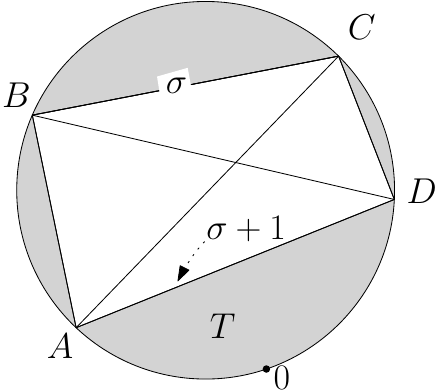}
                \caption{$\sigma_{BC}=\sigma$.}
                \label{fig:tri_type_mz}
\end{centering}
\end{subfigure}%
\caption{Writing $\sigma=\max(\sigma_{AB},\sigma_{BC},\sigma_{CD})$, the subfigures correspond to possible configurations arising in the proof of Proposition~\ref{prop:2_same}.}
\label{fig:constrained-types}
\end{figure}

%\rrr{Fig.~2: it is unclear why the figure shows the point E (what is this point?) and the triangle BEC. There is no mention of this point and this triangle in Proposition 2. If there is a reason for this, please mention it in the text. Otherwise, remove this from the figure.} \comment{DW: I agree. Louigi, please remove E and BEC from Figure 2.}

\begin{prop}\label{prop:three_equal}
If $\sigma_{AC}=\sigma_{BD}=\sigma_{BC}$ and $(\ell_{T'},m_{T'},r_{T'},z_{T'})=(\ell_T,m_T,r_T,z_T)$, 
then either $\sigma_{AD}=\sigma_{BC}$ or $\sigma_{AB}=\sigma_{BC}=\sigma_{CD}$. 
\end{prop}

\begin{proof}
In this case $ABC$ is type-$\ell$ or type-$z$, and $BCD$ is type-$m$ or type-$z$. 
If $ABC$ is type-$\ell$ then since the triangle type vector does not change when flipping, it must be that $ABD$ is type-$\ell$, which implies that $\sigma_{AD}=\sigma_{BD}$, yielding the result in this case since $\sigma_{BD} = \sigma_{BC}$. 
Similarly, if $BCD$ is type-$m$ then it must be that $ACD$ is type-$m$, which implies that $\sigma_{AD}=\sigma_{AC}$. Otherwise, both $ABC$ and $BCD$ are type-$z$, in which case we indeed have $\sigma_{AB}=\sigma_{BC}=\sigma_{CD}$. 
\end{proof}

For each triangulation $T$ in $\cA_n$ and $k \in \{1,2\}$ and $i\in\{0,1,\dots,\lceil{\log_\alpha(n-1)\rceil}\}$, let 
\[
s^k_i(T) := \#\{e \in E(T): Q(e)\mbox{ is type-}k,\, \sigma_e=i\}\, .
\]
Assign an integer label $c(T)$ to $T$ given by
\[
c(T) \;:=\; \left(\sum_{i=0}^{\lceil{\log_\alpha(n-1)\rceil}} \!\!\!\!\! 2i s^1_i(T) \;\;\; + \sum_{i=0}^{\lceil{\log_\alpha(n-1)\rceil}} \!\!\!\!\! 3is^2_i(T) \right) \mod (3\lceil \log_\alpha n\rceil)\, . 
\] 
The utility of such a labelling rule is explained by the following fact. 
We continue to work with triangulations $T$ and $T'$ related by an edge flip within quadrilateral $ABCD$ with $AC \in E(T)$ and $BD \in E(T')$, as above. 

\begin{prop} \label{prop_two_equal}
If exactly two of $\sigma_{AC}$, $\sigma_{BD}$ and  $\sigma_{BC}$ are equal, then 
$c(T') \ne c(T)$. 
\end{prop} 

\begin{proof}
First suppose that $BC$ is not a boundary edge, and 
let $V$ be the unique vertex of $T$ with $B<V<C$ adjacent to both $B$ and $C$. 
Note that $ABCD=Q_T(AC)$ is type-1 in $T$ and $ABCD=Q_{T'}(BD)$ is type-2 in $T'$. 
Also, $Q_T(BC) = ABVC$ is type-2 in $T$ and $Q_{T'}(BC)=BVCD$ is type-1 in $T'$. 
It is not hard to check that no other quadrilaterals change type when moving from $T$ to $T'$. Thus 
\begin{align*} 
c(T')-c(T) 
& = 3\sigma_{BD}-2\sigma_{AC} +2\sigma_{BC}-3\sigma_{BC} & \mod (3\lceil \log_\alpha n\rceil)\, \ \\
& = 3\sigma_{BD}-2\sigma_{AC} - \sigma_{BC} & \mod (3\lceil \log_\alpha n\rceil)\, . 
\end{align*}

It follows that if $\sigma_{BC} < \sigma_{BD}=\sigma_{AC}$, then 
\[
c(T')-c(T) = \sigma_{BD} - \sigma_{BC}  \mod (3\lceil \log_\alpha n\rceil) \ne 0\, ;
\]
the difference is non-zero modulo $(3\lceil \log_\alpha n\rceil)$ since all scales are at most $\lceil \log_\alpha n\rceil$. Similarly, if 
$\sigma_{BC} = \sigma_{BD}<\sigma_{AC}$ then 
\[
c(T')-c(T) = 2(\sigma_{AC}-\sigma_{BD}) \mod (3\lceil \log_\alpha n\rceil) \ne 0\, .
\]
Finally, if 
$\sigma_{AC}=\sigma_{BC}< \sigma_{BD}$ then (since $BD$ and $AC$ are non-boundary edges)
\[
c(T')-c(T) = 3(\sigma_{BD}-\sigma_{AC}) \mod (3\lceil \log_\alpha n\rceil) \ne 0\, 
\]
Since $\sigma_{BC} \le \min(\sigma_{AC},\sigma_{BD})$, these are the only possibilities. 
Here we use that $\sigma_{BD}\geq 1$ and $\sigma_{AC}\geq 1$ (since $D-B\geq 2$ and $C-A\geq 2$).

The case when $BC$ is a boundary edge is very similar, but easier. In this case, 
\[
c(T')-c(T) = 3\sigma_{BD}-2\sigma_{AC} \mod (3\lceil \log_\alpha n\rceil)\, .
\]
Since $BC$ is a boundary edge, $AC$ and $BD$ are not, so $\sigma_{BC}=0$ and $\sigma_{AC}\ne 0$ and $\sigma_{BD} \ne 0$. It follows by assumption that $\sigma_{AC}=\sigma_{BD}$, so 
\[
c(T')-c(T) = \sigma_{AC} \mod (3\lceil \log_\alpha n\rceil) \ne 0\, .\qedhere
\]
\end{proof}

%\sss{p. 4, l.~32--39: This should appear as a corollary so that you can explicitly refer to 	it later. Otherwise, some cross-references to these cases are way too far otherwise 	(for instance in p. 6, l. 9 and l. 49).}

Propositions~\ref{prop:2_same},~\ref{prop:three_equal} and~\ref{prop_two_equal} imply that the label $c(T)$ and the type vector 
$(\ell_T,m_T,r_T,z_T)$ together distinguish $T$ from $T'$ except in the following cases. 
\begin{enumerate}[(a)]
\item $AC$, $BD$, $BC$, and $AD$ have the same scale and $AB$, $CD$ have smaller scales. 
\item $AC$, $BD$, $BC$, $AD$ and $AB$ have the same scale and $CD$ has a smaller scale. 
\item $AC$, $BD$, $BC$, $AD$ and $CD$ have the same scale and $AB$ has a smaller scale. 
\item $AC$, $BD$, $BC$, $AB$ and $CD$ have the same scale and $AD$ has a larger scale. 
\item All six edges $AB$, $AC$, $AD$, $BC$, $BD$ and $CD$ have the same scale. 
\end{enumerate}
To handle cases (a), (b) and (c) we track two additional parameters, and show that 
the parity of one or both parameters is different for $T$ and $T'$. In case (d) we again prove there is a change of parity, but of a third, more complicated parameter. For case (e) we use induction. 

\begin{figure}[b]
\begin{subfigure}[b]{0.48\textwidth}
\begin{centering}
		\includegraphics[width=0.95\textwidth,page=1]{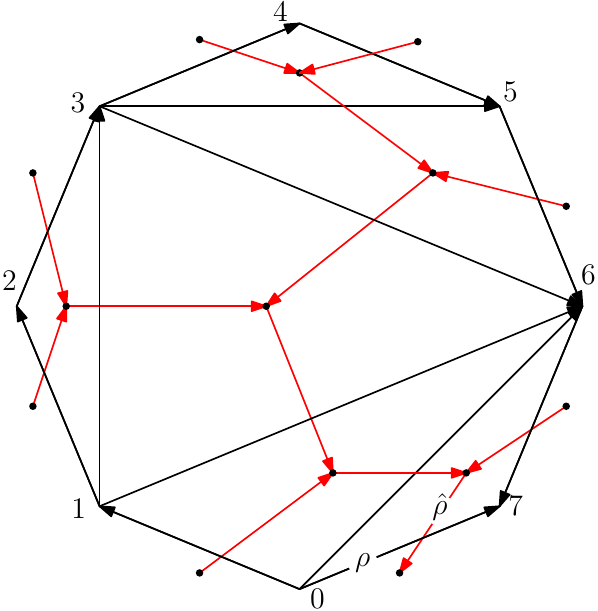}
                \caption{The dual tree of a triangulation of an 8-gon.}
                \label{fig:dual_oct}
\end{centering}
\end{subfigure}%
\quad
\begin{subfigure}[b]{0.48\textwidth}
\begin{centering}
		\includegraphics[width=0.95\textwidth,page=2]{dual_tree.pdf}
                \caption{The subgraph $\widehat{S}$ of $\widehat{T}$ corresponding to a subgraph $S$ of a triangulation $T$.}
                \label{fig:dual_subforest}
\end{centering}
\end{subfigure}%
\caption{The dual trees of an $8$-gon and of a sub-triangulation of a $12$-gon}
\label{fig:dual_trees}
\end{figure}

%\sss{p. 4, l. 44: ``Figure 3a should make the definitions of the current paragraph clear",
%but the third sentence introduces a vertex v which is not indicated on the picture.}
%
%\sss{ p. 4, l. 44--53: In other words, $\widehat{T}$ is the dual tree of $T$ oriented towards its root.
%	Note that if you had drawn the polygon with its boundary edge $(0; n-1)$ on top
%	and all its vertices below from left to right, your tree would be much nicer drawn
%	(in particular oriented from bottom to top).}

Figure~\ref{fig:dual_oct} should make the following definitions clear. Orient the edges of the triangulation $T$ so that the head of each edge has larger label. The \defn{root edge} of $T$ is the edge $\rho=(0,n-1)$. 
%Then let $\widehat{T}$ be the dual tree of $T$: for each edge $e$ of $T$ there is a unique edge $\hat{e}$ of $\widehat{T}$ crossing $e$ and oriented from the left to the right of $e$ (when following $e$ from tail to head). The tree $\widehat{T}$ is rooted at the dual edge $\hat{\rho}$, whose head is the unique node of $\widehat{T}$ with out-degree zero. More formally: Augment $T$ by adding a vertex $v$ to the unbounded face, and join it to all vertices of the polygon. Let $\widehat{T}_0$ be the planar dual of the augmented graph; then $\widehat{T}$ is formed by removing all edges of $\widehat{T}_0$ lying entirely within the unbounded face of $T$, and orienting edges as described above. We remark that the triangulation $T$ is uniquely determined by $\widehat{T}_0$ and thus by $\widehat{T}$.
Now construct the following oriented tree $\widehat{T}$. First, augment $T$ by adding a vertex $v$ to the unbounded face, and join it to all vertices of the polygon. Let $\widehat{T}_0$ be the planar dual of the augmented graph; then $\widehat{T}$ is formed from $\widehat{T}_0$ by removing all edges of $\widehat{T}_0$ lying entirely within the unbounded face of $T$. For each edge $e$ of $T$ there is a unique edge $\hat{e}$ of $\widehat{T}$ crossing $e$. Orient $\hat{e}$ from the left to the right of $e$ (when following $e$ from tail to head). Root $\widehat{T}$ at the edge $\hat{\rho}$, whose head is the unique node of $\widehat{T}$ with out-degree $0$. Note that $\widehat{T}$ is a tree, which we call the \defn{dual tree} of $T$. 

Given an edge $e=UV$ of $T$ with $e \ne \rho$, the triangle containing the head of $\hat{e}$ is incident to both $U$ and $V$; let $W$ be its third node. By the above choice of orientation for $\hat{e}$, either $W < \min(U,V)$ or $W > \max(U,V)$. In the first case, $\hat{e}$ is a \defn{left turn}, and in the second case it is a \defn{right turn}. 

%\rrr{Page 4, line 55: when you write that "$W < \min(V, U)$ or $W > \max(V, U)$", please explain why. This is due to the choice of orientations for the edges of $T$ and $\widehat{T}$, which can be understood with some thought, but just mentioning that would help.} \comment{DW: done}

Given a subgraph $S$ of triangulation $T$, as illustrated in Figure~\ref{fig:dual_subforest}, let $\widehat{S}$ be the ``dual'' subgraph of $\widehat{T}$ with the same vertex set as $\widehat{T}$ and with edge set 
\[
E(\widehat{S}) \,:=\; \{\hat{e}: e \in E(S)\}.
\]
A node of $\widehat{S}$ is a \defn{leaf} if it has degree $1$. 
For each node $v$ of $\widehat{S}$, let $g_S(v)$ and $d_S(v)$ be defined as follows. 
Write $r$ for the root (the unique node of out-degree $0$) of the tree component of $\widehat{S}$ containing $v$. Then $g_S(v)$ and $d_S(v)$ are the number of left- and right-turns on the path from $v$ to $r$, respectively; see Figure~\ref{fig:lr_labeling}.  (Figure~\ref{fig:reduced_dualtree} is used later in the section.)\

\begin{figure}[b]
\begin{subfigure}[b]{0.48\textwidth}
\begin{centering}
\includegraphics[width=\textwidth,page=3]{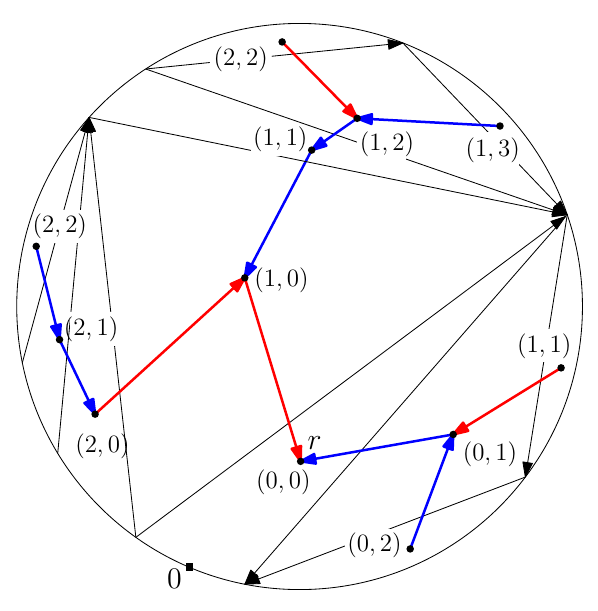}
\caption{The left-turn and right-turn labelling of a component $\hat{H}$ of $\widehat{S}$ with root node $r$. Labels are given in the form $(g(v),d(v))$ for all nodes $v$ of $\widehat{S}$.}
              \label{fig:lr_labeling}
\end{centering}
\end{subfigure}%
\quad
\begin{subfigure}[b]{0.48\textwidth}
\begin{centering}
\includegraphics[width=\textwidth,page=4]{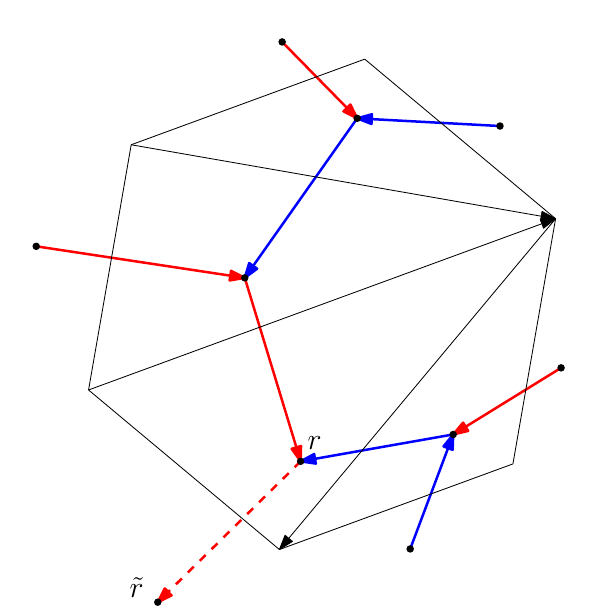}
\caption{The reduced tree $\tilde{H}$ corresponding to the component $\hat{H}$, together with the triangulation of a polygon to which $\tilde{H}$ is dual. The root edge of $\tilde{H}$ is dashed.}
              \label{fig:reduced_dualtree}
\end{centering}
\end{subfigure}%
\caption{In both subfigures, left-turn edges are red and right-turn edges are blue.}
\end{figure}

%\rrr{Page 5, Fig.~4: I'm not sure I understand this figure: for instance the red arc on the bottom-right of Fig.~4(a) starts from a vertex labelled (2,1) and ends in a vertex labelled (0,1), suggesting that both g and d decrease along this arc, which seems incompatible with the definitions. I'm possibly missing something here, but in any case, the figure needs to be explained more clearly in the text. ... Also, could you write the vertex labels in that figure?} \comment{DW: in Figure~4(A), the vertex on the right labelled (2,1) should be labelled (1,1), and the vertex at the bottom labelled (1,1) should be labelled (0,2). }

Recall that $T'$ is obtained from $T$ by flipping edge $AC$ within quadrilateral $ABCD$. For each $i\in\{1,2,\dots,\lceil \log_\alpha n \rceil\}$, let $S_i$ be the subgraph of $T$ with edge set 
$E(S_i) = \{e \in E(T): \sigma_e=i\}$, and let $S_i'$ be the subgraph of $T'$ with edge set 
$E(S_i') = \{e \in E(T'): \sigma_e=i\}$. Define $\widehat{S}_i$ and $\widehat{S}'_i$ as in the preceding paragraph (so $\widehat{S}_i$ is a subgraph of $\widehat{T}$ and $\widehat{S}_i'$ is a subgraph of $\widehat{T}'$). Let 
\[
G(T) := \sum_{i=1}^{\lceil \log_\alpha n \rceil} \sum_{v \in V(\widehat{S}_i)}
g_{S_i}(v) 
\quad \mbox{ and }\quad 
D(T) := \sum_{i=1}^{\lceil \log_\alpha n \rceil} \sum_{v \in V(\widehat{S}_i)} 
d_{S_i}(v)\, .
\]

The following proposition implies that in cases (a), (b) and (c), flipping edge $AC$ yields a change in parity of at least one of $G$ and $D$. 
\begin{prop}\label{prop_parity_cases123}
If the scales of the edges $AB$, $AC$, $AD$, $BC$, $BD$ and $CD$ are as in cases (a), (b) or (c) above, then $G(T')=G(T)-1$ or $D(T')=D(T)+1$ (or both).
\end{prop}
\begin{proof}
Let $\sigma=\sigma_{AC}$.
We first claim that for all $i \ne \sigma$, the contributions to $G(T)$ and to $D(T)$ from scale-$i$ nodes are unchanged by the edge flip operation; that is, 
\begin{equation}\label{eq:parity123}
\sum_{i \ne \sigma}\!\sum_{v \in V(\widehat{S}_i)}\!
g_{S_i}(v)=\sum_{i \ne \sigma}\!\sum_{v \in V(\widehat{S}_i')}\!
g_{S_i'}(v)
\quad \mbox{ and }\quad 
\sum_{i \ne \sigma}\!\sum_{v \in V(\widehat{S}_i)} \!
d_{S_i}(v)=\sum_{i \ne \sigma}\! \sum_{v \in V(\widehat{S}_i')} \!
d_{S_i'}(v)\, .
\end{equation}
We prove these equalities in case (a); the other two cases are similar but easier. 

The triangle containing the head of the edge $\hat{e}_{AB}$ dual to $AB$ is $ABC$ in $T$ and is $ABD$ in $T'$. The case (a) assumptions on the scales of the edges then imply that the head of $\hat{e}_{AB}$ has out-degree $0$ and in-degree $1$ in $\widehat{S}_{\sigma_{AB}}$. In particular, it is the root of its component of $\widehat{S}_{\sigma_{AB}}$. Moreover, $\hat{e}_{AB}$ is a right-turn edge in both $T$ and $T'$, since $C$ and $D$ are both larger than $A$ and $B$. 

%\rrr{Page 6, in the proof of Proposition 5, line 25: when you write that "$\hat{e}_{AB}$ is a right-turn edge" you actually mean that it is a right-turn edge in both T AND T', right? (Since you tell that BOTH C and D are larger than A and B.)}

%\rrr{Page 6, in the proof of Proposition 5, lines 26-28: when you write that "the triangle containing the tail of the edge $\hat{e}_{CD}$ dual to CD is ACD in T and is BCD in T'", shouldn't this be "the head of $\hat{e}_{CD}$" instead of its tail? Also, in the assertion "$\hat{e}_{CD}$ has in-degree 1 and out-degree 0 within $\widehat{S} \sigma_{CD}$", shouldn't it be "the head of $\hat{e}_{CD}$" instead of just "$\hat{e}_{CD}$"?}

Similarly, the triangle containing the head of the edge $\hat{e}_{CD}$ dual to $CD$ is $ACD$ in $T$ and is $BCD$ in $T'$. The assumptions on the scales of edges again imply that the head of $\hat{e}_{CD}$ has in-degree $1$ and out-degree $0$ within $\widehat{S}_{\sigma_{CD}}$, so is the root of its component of $\widehat{S}_{\sigma_{CD}}$. Moreover, $\hat{e}_{CD}$ is a left-turn edge  in both $T$ and $T'$, since $A$ and $B$ are both smaller than $C$ and $D$. 

%\rrr{Page 6, in the proof of Proposition 5, line 29 (or 30): "$\hat{e}_{AB}$ is a left-turn edge" should be replaced by "$\hat{e}_{CD}$ is a left-turn edge", right? In addition, please tell that it is a left-turn edge in both T and T', if that's what you mean (because you write "A and B are both smaller than C and D".)}

Since the structures of $T$ and of $T'$ are unaffected outside of the quadrilateral $ABCD$, the equalities in (\ref{eq:parity123}) follow in case (a). 

We now restrict our attention to the scale $\sigma$. We write $g(\cdot) = g_{S_\sigma}(\cdot)$ and $d(\cdot) = d_{S_\sigma}(\cdot)$, 
and likewise $g'(\cdot) = g_{S'_\sigma}(\cdot)$ and $d'(\cdot) = d_{S'_\sigma}(\cdot)$. Note that all nodes not lying within the quadrilateral $ABCD$ either belong to both $S_\sigma$ and $S'_{\sigma}$ or belong to neither of $S_\sigma$ and $S'_{\sigma}$. 

The remainder of the proof boils down to inspection of Figures~\ref{fig:lr_labels_1},~\ref{fig:lr_labels_2} and~\ref{fig:lr_labels_3}. 
In case (a), observe (see Figure~\ref{fig:lr_labels_1}) that $g(u)=g'(u)=a+1$ and $d(u)=d'(u)=b+1$, which implies that $(g(q),d(q))=(g'(q),d'(q))$ for all nodes $q$ not lying within $ABCD$. Since $g(v)+g(x)=2a+1=g(p)+g(z)+1$ and $d(v)+d(z)=2b=d(p)+d(z)-1$, it follows that $G(T)=G(T')+1$ and $D(T')=D(T)-1$. 

%\rrr{Page 6, in the proof of Proposition 5, line 41: shouldn't it be "$d(w)=d'(w)$" instead of "$d(y)=d'(y)$"? In particular, there is no vertex $y$ in Fig.~6.}

Figure~\ref{fig:lr_labels_2} depicts the situation in case (b). In this case $d(u)=d'(u)$ and $d(w)=d'(w)$, which implies that $d(q)=d'(q)$ for all $q$ not lying in $ABCD$. 
Since $d'(z)+d'(p)=d(v)+d(x)+1$, it follows that $D(T')=D(T)+1$.

Finally, Figure~\ref{fig:lr_labels_3} relates to case (c). In this case $g(u)=g'(u)$, $g(y)=g'(y)$, and 
$g(v)+g(x)=g'(z)+g(p)+1$, so the above argument implies that $G(T')=G(T)-1$. 
%\rrr{Page 6, in the proof of Proposition 5, line 45 (or 44): shouldn't it be "g(y)=g'(y)" instead of "g(y)=g'(t)"?} 
\end{proof}

\begin{figure}[ht]
\begin{centering}
\includegraphics[width=0.85\textwidth,page=2]{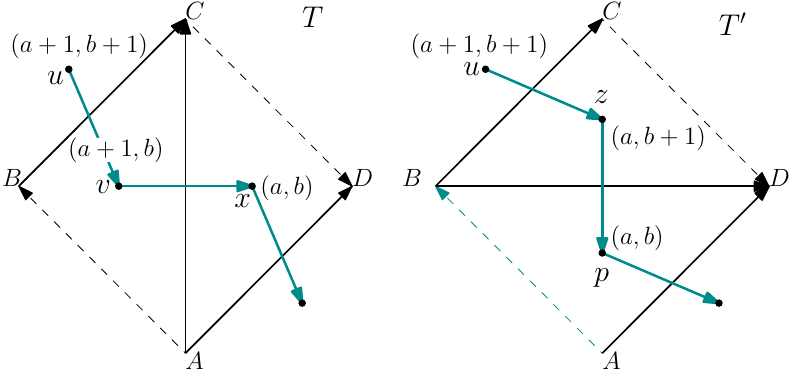}
\caption{The left-turn and right-turn labels near quadrilateral $ABCD$ in $T$ and $T'$: case (a). Here $(g(x),d(x))=(a,b)$, $(g(v),d(v))=(a+1,b)$ and $(g(u),d(u))=(a+1,b+1)$.}
             \label{fig:lr_labels_1}
\end{centering}
\end{figure} 
\begin{figure}[ht] 
\begin{centering} 
\includegraphics[width=0.85\textwidth,page=3]{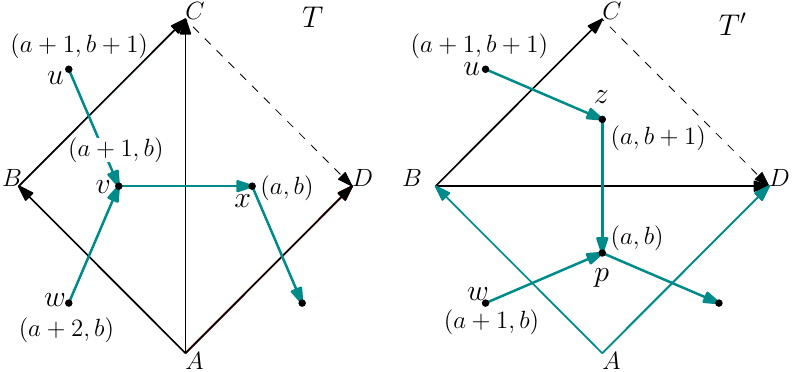} 
\caption{The left-turn and right-turn labels near $ABCD$ in $T$ and $T'$: case (b).} 
              \label{fig:lr_labels_2} 
\end{centering} 
\end{figure} 
\begin{figure}[ht] 
\begin{centering} 
\includegraphics[width=0.85\textwidth,page=4]{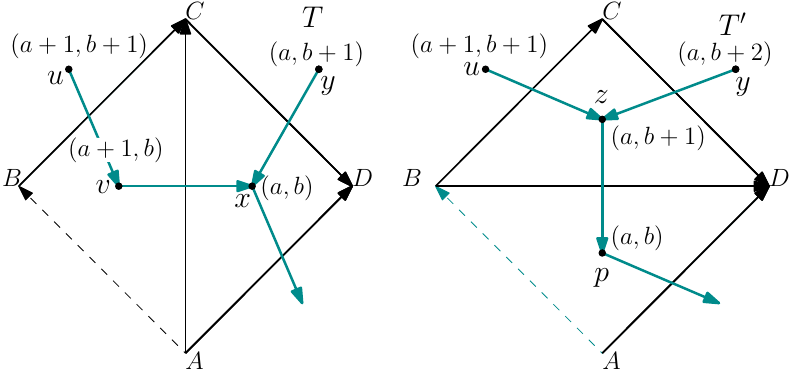} 
\caption{The left-turn and right-turn labels near $ABCD$ in $T$ and $T'$: case (c).}
              \label{fig:lr_labels_3} 
\end{centering} 
\end{figure} 

%\rrr{Page 6, line 50: "a rooted binary subtree" would sound better than "a rooted sub-binary tree".}

%\rrr{Page 6, paragraph above Proposition 6: how do we know that $\tilde{H}$ remains connected after the nodes of degree two have been removed? More precisely, is it possible that some vertex of $\hat{H}$ of degree 3 is separated from $\tilde{r}$ by degree 2 nodes? (Which would then disconnect $\tilde{H}$.) Could you write a short argument (just one sentence would be enough) to explain that? In fact, the explanation of how $\tilde{H}$ is built from $\hat{H}$ may need to be rewritten in a clearer way (you also remove the leaves of $\hat{H}$, whose degree is not two, but one).}\comment{DW: ignore}

We now turn our attention to cases (d) and (e). 
Consider any subgraph $S$ of $T$, and let $\hat{H}$ be a connected component of $\widehat{S}$. Note that $\hat{H}$ is a rooted sub-binary tree (that is, every node has degree at most three; see Figure~\ref{fig:dual_subforest}). Let $\tilde{H}$ be the tree obtained from $\hat{H}$ as follows (see Figure~\ref{fig:lr_labeling} and~\ref{fig:reduced_dualtree}). First, if the root $r$ of $\hat{H}$ has exactly two children then add a new node $\tilde{r}$ incident only to $r$ and reroot at $\tilde{r}$. Next, suppress each node of degree exactly two (that is, contract one edge incident to the node). We obtain a rooted binary tree $\tilde{H}$ called the \defn{reduced tree} of $\hat{H}$. 

%\rrr{Page 6, paragraph above Proposition 6: this paragraph explains why you drew Fig. 4 as you did. So in relation with my above comment 7), I suggest that you make two figures instead of Fig. 4. The first should describe the functions g and d, be placed in Page 5, and referred to there. In this figure, you should use the same triangulation as in Fig. 3 (and actually, in Fig.~3 itself, you may use only one triangulation instead of different triangulations for (a) and (b)) and provide the vertex labels. The other figure should explain the construction of H~ and be placed in Page 6. In this other figure, which should look like the current Fig.~4, please make sure that the relation between (a) and (b) is clear (for example, vertex labels would help.)} \comment{DW: added a comment}

\begin{prop}\label{prop:leaf_count}
For each $i\in\{1,2,\dots,\lceil \log_\alpha n \rceil\}$ and for each component $\hat{H}$ of $\widehat{S}_i$, the reduced tree $\tilde{H}$ has at most $2\alpha-1$ leaves.
\end{prop}

\begin{proof}
Fix any node $u$ of $\tilde{H}$ with in-degree zero, and consider the edge $uv$ incident to $u$ in $\hat{H}$. Then $uv$ is dual to an edge $AB$ with $\sigma_{AB}=i$ so with $\alpha^{i-1} < B-A \le \alpha^i$. Now fix another node $w$ of $\tilde{H}$ with in-degree zero, write $wx$ for the edge incident to $w$ in $\hat{H}$, and let $CD$ be its dual edge. Then necessarily either $A < B \leq C < D$ or $C < D \leq B < A$. 

Consider an oriented edge $yr$ where $r$ is the root of $\hat{H}$. Writing $EF$ for the edge dual to $yr$, then the observation of the preceding paragraph implies that $F-E > \alpha^{i-1} \cdot \ell$, where $\ell$ is the number of nodes with in-degree zero in the subtree rooted at $y$. On the other hand, $\sigma_{EF}=i$ so $F-E \le \alpha^i$; so $\ell \le \alpha-1$. If $r$ has only one child (so is a leaf itself) this yields that $\tilde{H}$ has at most $\alpha$ leaves. If $r$ has two children then each of their subtrees contains at most $\alpha-1$ leaves; in this case $\tilde{r}$ is also a leaf, so the total number of leaves is at most $2(\alpha-1)+1$. 
\end{proof}

For any subgraph $S$ of the triangulation $T$, the embedding of $\widehat{T}$ in the plane induces a total order of the connected components of $\widehat{S}$, given by the order their roots are visited by a clockwise tour around the contour of $\widehat{T}$ starting from the head of the root edge $\rho$. For each $1 \le i \le \lceil \log_\alpha n \rceil$, list the components of $\widehat{S}_i$ in the order just described as $H_{i,1},\ldots,H_{i,\ell}$, where $\ell=\ell(T,i)$ is the number of such components. Then, for $1 \le j \le \ell(T,i)$ let $\tilde{H}_{i,j}$ be the reduced tree of $H_{i,j}$. 

Each tree from $(\tilde{H}_{i,j},i \le \lceil \log_\alpha n\rceil,j \le \ell(T,i))$ is a dual to a unique triangulation $\tilde{T}_{i,j}$ of a polygon, as in Figure~\ref{fig:reduced_dualtree}.  Proposition~\ref{prop:leaf_count} implies that $\tilde{T}_{i,j}$ belongs to an associahedron $\cA_k$ for some $k \le 2\alpha-1$. Let $\phi$ be a proper colouring of the disjoint union of $(\cA_k,k \le 2\alpha-1)$, with colours $\{1,\ldots,\chi(\cA_{2\alpha-1})\}$, and define 
\[
I(T) := \left(\sum_{i=0}^{\lceil \log_\alpha n \rceil} \sum_{j=1}^{\ell(T,i)} \phi(\tilde{T}_{i,j}) \right)\mod \chi(\cA_{2\alpha-1})\, .
\]

\begin{prop}\label{prop_induction}
In cases (d) and (e), we have $I(T) \ne I(T')$. 
\end{prop}

% \rrr{Page 7, figures 5-7: you should use the red/blue coloring in this figure as you did in Fig.~4.}\comment{DW: Louigi to decide}

\begin{proof}
Write $vx=\hat{e}$ for the dual edge of $AC$ in $T$, and $zp=\hat{e}_{BD}$ for the dual edge of $BD$ in $T'$. Figures~\ref{fig:case_4} and~\ref{fig:case_5} illustrate cases (d) and (e) respectively. 

\begin{figure}[ht]	
	\begin{centering}
		\includegraphics[width=0.85\textwidth,page=5]{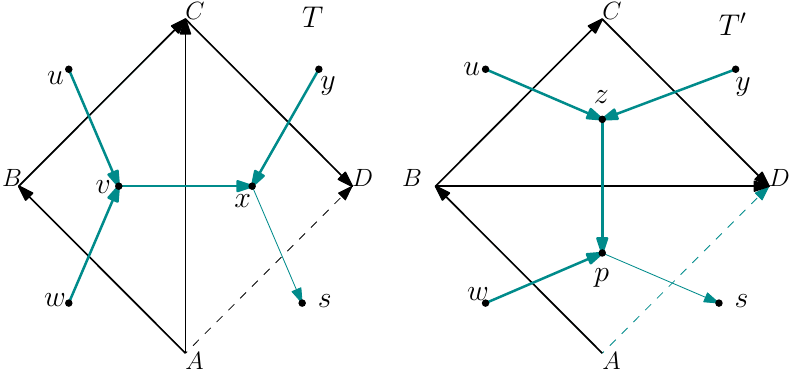}
		\caption{The structure near the quadrilateral $ABCD$ in $T$ and $T'$ case (d). The dashed edges have scale $\sigma_{AD} > \sigma$, all other edges of the triangulations shown in the figure have scale $\sigma$.}
		\label{fig:case_4}
	\end{centering}
\end{figure} 
\begin{figure}[ht]
	\begin{centering}
		\includegraphics[width=0.85\textwidth,page=6]{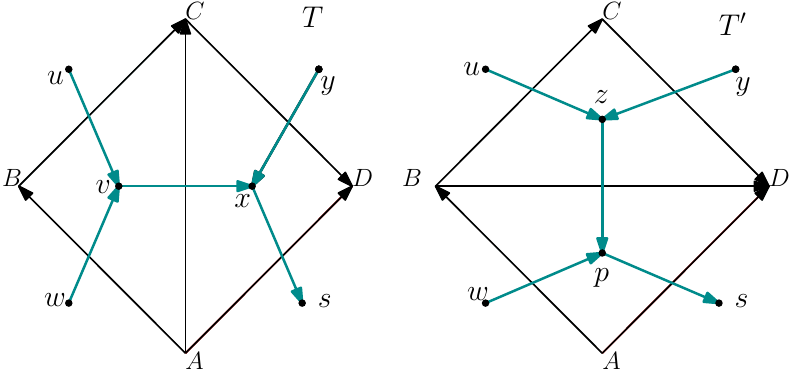}
		\caption{The structure near the quadrilateral $ABCD$ in $T$ and $T'$ in case (5). All edges of the triangulations shown in the figure have scale $\sigma$.}
		\label{fig:case_5}
	\end{centering}
\end{figure}

The clockwise contour exploration of a rooted plane tree is a walk around the outside of the tree which starts and finishes at the root, keeping the unbounded face to its left at all times. This walk traverses each edge exactly twice, and records the vertices it visits in sequence, with repetition. 
In cases (d) and (e), for the clockwise contour explorations of $\widehat{T}$ and of $\widehat{T}'$, there are (possibly empty) strings $P_1,\ldots,P_5$ so that the sequences recorded by the contour explorations of $\widehat{T}$ and of $\widehat{T}'$ are respectively of the form 
\[
P_1sxvwP_2wvuP_3uvxyP_4yxsP_5 \quad \text{and} \quad 
P_1spwP_2wpzuP_3uzyP_4yzpsP_5\,;
\]
see Figures~\ref{fig:case_4} and~\ref{fig:case_5}.

The contour explorations of $\widehat{T}$ and $\widehat{T}'$ agree until they visit dual vertices lying within $ABCD$. It follows that if $H_{\sigma,j}$ is the component of $S_\sigma$ containing $\hat{e}_{AC}$, then the component of $S_{\sigma}'$ containing $\hat{e}_{BD}$ is $H'_{\sigma,j}$. 

In case (d), since $x$ has two children in $H_{\sigma,j}$, by construction it is the unique child of the root of $\tilde{H}_{\sigma,j}$. It is thus natural to identify $s$ with the root of $\tilde{H}_{\sigma,j}$. We may likewise identify $s$ with the root of $\tilde{H}'_{\sigma,j}$, since $p$ has two children in $H'_{\sigma,j}$. After the addition of $s$ as a root, the nodes $v,x,p$ and $z$ all have degree $3$, so none of these nodes are suppressed when constructing $\tilde{H}_{\sigma,j}$ and $\tilde{H}_{\sigma,j}'$ from $H_{\sigma,j}$ and $H_{\sigma,j}'$. 

In case (e), nodes $v,x,p$ and $z$ all have degree $3$, and the edges $xs$ and $ps$ belong to $H_{\sigma,j}$ and $H_{\sigma,j}'$, respectively.

It follows from the two preceding paragraphs that in cases~(d) and~(e), we may view $v$ and $x$ as nodes of both $H_{\sigma,j}$ and $\tilde{H}_{\sigma,j}$, and $p$ and $z$ as nodes of both $H_{\sigma,j}'$ and $\tilde{H}_{\sigma,j}'$. It is then clear that 
flipping the edge $AC$ in the triangulation $T$ corresponds precisely to flipping the corresponding edge in $\tilde{T}_{\sigma,j}$ to form $\tilde{T}'_{\sigma,j}$. 

Since $\tilde{T}_{\sigma,j}$ and $\tilde{T}_{\sigma,j}'$ are related by a single edge flip, and $\phi$ is a proper colouring, it follows that 
$\phi(\tilde{T}_{\sigma,j}) \ne \phi(\tilde{T}_{\sigma,j}')$. Since all other components of $S_\sigma$, and more generally of each $(S_i,1 \le i \le \lceil \log_\alpha n \rceil)$, are unchanged when moving from $T$ to $T'$, the result follows. 
\end{proof}

\begin{proof}[Proof of Theorem~\ref{thm:main}.] 
	Consider the graph $G$ with vertex set $\N^4$ where distinct vertices $(\ell,m,r,z)$ and $(\ell',m',r',z')$ are adjacent if 
	$|\ell'-\ell|+|m'-m|+|r'-r|+|z'-z| \le 4$. 
	This graph has maximum degree at most $320$ (see \citep{A008412}), so is $321$-colourable. (We are not optimizing constants.)\  
	Let $\kappa:\N^4 \to \{1,\ldots,321\}$ be a proper colouring of~$G$. 
	
	For each triangulation $T$ of $\cA_n$, let $g(T) := G(T) \bmod 2$ and $d(T) :=D(T)\bmod 2$; 
	colour $T$ by the 5-tuple 
	\begin{equation}
		\label{TheColouring}
		\psi(T) := \big(\kappa((\ell_T,m_T,r_T,z_T)),\, c(T),\, g(T),\, d(T),\, I(T) \big).
	\end{equation}
	We now show that $\psi$ is a proper colouring of $\cA_n$. 
	
	Consider adjacent vertices $T$ and $T'$ in $\cA_n$. 
	In all situations covered by Proposition~\ref{prop:2_same}, the vectors $(\ell_{T'},m_{T'},r_{T'},z_{T'})$ and $(\ell_T,m_T,r_T,z_T)$ are different, and thus adjacent in $G$, 
	since each of $|\ell_{T'}-\ell_T|$,  $|m_{T'}-m_T|$, $|r_{T'}-r_T|$ and $|z_{T'}-z_T|$ is at most 1. 
	Thus $\kappa((\ell_{T'},m_{T'},r_{T'},z_{T'})) \ne \kappa((\ell_T,m_T,r_T,z_T))$. 
	In the situations covered by Proposition~\ref{prop_two_equal}, we have $c(T) \ne c(T')$. 
	By Proposition~\ref{prop_parity_cases123}, in cases (a)--(c), either $g(T) \ne g(T')$ or $d(T) \ne d(T')$ or both. Finally, in cases (d) and (e), by Proposition~\ref{prop_induction} we have $I(T) \ne I(T')$. 
	
	Thus $\psi(T)\ne \psi(T')$, implying for any integer $\alpha \ge 3$, 
	\[
	\chi(\mathcal{A}_n) \le 321\cdot \lceil 3 \log_\alpha n \rceil \cdot 2 \cdot 2 \cdot \chi(\mathcal{A}_{2\alpha-1}).
	\]
	Taking $\alpha=3$ yields $\chi(\mathcal{A}_{2\alpha-1}) = \chi(\mathcal{A}_5)=3$, since $\mathcal{A}_5$ is a $5$-cycle. 
	It follows that 
	\[
	\chi(\mathcal{A}_{n}) \le 12\cdot 321 \cdot \lceil 3 \log_3 n \rceil \in O(\log n)\, .\qedhere
	\]
\end{proof}

\subsection*{Acknowledgements}
This research was initiated at the 2018 Barbados Workshop on Graph Theory, held at the Bellairs Research Institute of McGill University.

%%%%%%%%%%%%%%%%%%%%%%%%%%%%
% BIBLIOGRAPHY
%%%%%%%%%%%%%%%%%%%%%%%%%%%%

\end{document}